\documentclass[amstex,12pt]{article}

\usepackage{graphicx}
\usepackage{amsfonts}
\usepackage{amssymb}
\usepackage{amsmath}

\def\R{\mathbb{R}}
\def\e{{\varepsilon}}

\def\rad{\mathrm{rad}}

\def\Rad{\mathrm{Rad}}
\def\nm#1{\left\Vert#1\right\Vert}

\newtheorem{thm}{Theorem}
\newtheorem{lm}{Lemma}

\newtheorem{co}{Corollary}
\newtheorem{prop}{Proposition}

\newtheorem{conj}{Open Problem}

\newenvironment{proof}[1][Proof]{\noindent\textsl{#1:} }{\hfill $\Box$}

\input{Macros.tex}

\begin{document}

\title{Some inequalities for tetrahedra}
\author{Jin-ichi Itoh, Jo\"{e}l Rouyer and Costin V\^{\i}lcu}
\maketitle

\begin{abstract}
We prove inequalities involving intrinsic and extrinsic radii and diameters of tetrahedra. 
\end{abstract}

%%%%%%%%%%%%%%%%%%%%%%%%%%%%%%%%%%%%%%%%%%%%%%%%%

\section{Introduction}

A {\it convex surface} $S$ is the boundary of a {\it convex body}
(compact convex set with interior points) in the Euclidean space $\R^3$,
or a doubly covered planar convex body; in the latter case it is called {\it degenerate}.
Denote by ${\cal S}$ the set of all convex surfaces.

The intrinsic metric $\rho$ of a convex surface $S$ is defined, for any points $x,y$ in $S$,
as the length $\rho (x,y)$ of a {\it geodesic segment} (\textit{i.e.}, shortest path on $S$)
joining $x$ to $y$.

Denote by $\mathrm{Diam}\left( S\right)$ the intrinsic diameter of
$S \in {\cal S}$, and by $\mathrm{diam}\left(S\right)$ its extrinsic
diameter,
$$\mathrm{Diam}\left( S\right) := \max_{x,y \in S} \rho(x,y), \;\;\;\;\;
\mathrm{diam}\left( S\right): = \max_{x,y \in S}||x-y||.$$

Denote by $\mathrm{Rad}\left(S\right)$ the intrinsic radius of
$S \in {\cal S}$, and by $\mathrm{rad}\left(S\right)$ its extrinsic
radius,
$$\mathrm{Rad}\left( S\right) := \min_{x \in S} \max_{y\in S} \rho(x,y), \;\;\;\;
\mathrm{rad}\left( S\right) := \min_{x \in S} \max_{y \in S}||x-y||.$$

The first three quantities introduced above proved useful for the study of convex surfaces, as one can briefly see in the following. 
But the fourth one, $\mathrm{rad}\left( S\right)$, seems somehow neglected, despite its natural definition.

N. P. Makuha \cite{ma} showed that 
$${\rm  Diam} (S) \leq \frac\pi2 {\rm  diam} (S)$$
holds for any convex surface $S$, with equality if and only if $S$ is a surface of revolution having constant width, 
see for instance \cite{cg} for definition and fundamental properties.

On the other hand, clearly 
$${\rm rad} (S) \leq {\rm  diam} (S) \leq 2 {\rm rad}S,$$ 
and the surfaces satisfying ${\rm rad} (S) = {\rm  diam} (S)$ have constant width.

One also has 
$${\rm Rad} (S) \leq {\rm  Diam} (S) \leq 2 {\rm Rad}(S),$$
and the surfaces satisfying ${\rm Rad} (S) = {\rm  Diam} (S)$ are studied in \cite{vz1}, while
those satisfying ${\rm  Diam} (S) = 2 {\rm Rad} (S)$ are studied in \cite{vz2} in relation to critical points for distance functions.

The space $\mathcal{T}$ of all the tetrahedra in $\mathbb{R}^{3}$, up to isometry and homothety, 
is studied in \cite{Ro-V} with respect to the number of local maxima of intrinsic distance functions.

The intrisic diameter and radius of a regular tetrahedron are computed by J. Rouyer in \cite{ro1}, while
V. Dods, C. Traub, and J. Yang \cite{dty} studied geodesics on the regular tetrahedron.

An old conjecture of A. D. Aleksandrov states that a convex surface with unit intrinsic diameter and largest area is a doubly covered disk.
V. A. Zalgaller \cite{Zalgaller} proved that among all tetrahedra with unit intrinsic diameter, 
only the regular tetrahedron with edges of length $\sqrt3 /2$ has the largest surface area, which is equal to $3\sqrt3 /4$.

In this note we prove new inequalities involving radii and diameters of tetrahedra.
In the next section we present some necessary preliminaries.
The main result in Section \ref{del} (Theorem \ref{TE1}) concerns the ratio $\frac{\mathrm{Diam}}{\mathrm{diam}}$, 
while the main result in Section \ref{radi} (Theorem \ref{Tder}) treats the ratio $\frac{\mathrm{Rad}}{\mathrm{diam}}$, 
both considered for $T \in \mathcal{T}$.
The division of our results into two main sections is just to ease the reading, their topics obviously overlap.

Concluding, we have the following inequalities for tetrahedra, either generally known or proven here; to ease the presentation, the functions are given without the argument
$T\in\mathcal{T}$.

\[
1\leq\frac{\mathrm{Diam}}{\mathrm{diam}}\leq\frac{2}{\sqrt{3}}, \hspace{0.5cm}
1 < \frac{\mathrm{Diam}}{\mathrm{Rad}}\leq2 , \hspace{0.5cm}
1 < \frac{\mathrm{diam}}{\mathrm{rad}} \leq 2 , \hspace{0.5cm}
\frac{\mathrm{Rad}}{\mathrm{diam}} \leq 1 \text{,}
\]
\[
1 \leq \frac{\mathrm{Rad}}{\mathrm{rad}} < 2, \hspace{0.5cm}
\frac{\sqrt{3}}{4} < \frac{\mathrm{rad}}{\mathrm{Diam}} < 1 \text{.}
\]

Of the above inequalities, some are not sharp and could be improved; for example,
our Open Problem \ref{conj} asks to prove (or disprove) that
\[
\frac2{\sqrt{3}} \leq \frac{\mathrm{Diam}\left( T\right) }{\mathrm{Rad}\left( T\right) }\text{.}
\]

%%%%%%%%%%%%%%%%%%%%%%%%%%%%%%%%%%%%%%%%%%%%%%%%%

\section{Preliminaries}

Let $P$ be (the surface of) a convex polyhedron. 

A {\em geodesic segment} on a $P$ is a shortest path between its extremities.

The {\em cut locus} $C(x)$ of the point $x$ on $P$ is the set of endpoints (different from $x$) of all nonextendable geodesic segments (on the surface $P$) starting at $x$.
Equivalently, it is the closure of the set of all those points $y$ to which there is more than one shortest path on $P$ from $x$.

The following lemma presents several known properties of cut loci on convex polyhedra, see e.g. \cite{code1}.

\begin{lm}
\label{basic}
(i) $C(x)$ has the structure of a finite $1$-dimensional simplicial complex which is a tree.
Its leaves (endpoints) are vertices of $P$, and all vertices of $P$, excepting $x$ (if the case), are included in $C(x)$.
All vertices of $P$ interior to $C(x)$ are considered as junction points.
 
(ii) Each point $y$ in $C(x)$ is joined to $x$ by as many geodesic segments as the number of connected
components of $C(x) \setminus {y}$.
For junction points in $C(x)$, this is precisely their degree in the tree.

(iii) The edges of $C(x)$ are geodesic segments on $P$.

(iv) Assume the geodesic segments $\gamma$ and $\gamma'$ from $x$ to $y \in C(x)$ are bounding a domain $D$ of $P$, 
which intersects no other geodesic segment from $x$ to $y$.
Then there is an arc of $C(x)$ at $y$ which intersects $D$ and it bisects the angle of $D$ at $y$.
\end{lm}

We shall implicitely use Alexandrov's Gluing Theorem stated below, see \cite{code2}, p.100.

\begin{lm}
\label{gluing}
Consider a topological sphere $S$ obtained by gluing planar polygons (i.e., naturally identifying pairs of sides of the same length) such that 
at most $2\pi$ angle is glued at each point.
Then $S$, endowed with the intrinsic metric induced by the distance in $\R^2$, 
is isometric to a polyhedral convex surface $P\subset\R^3$, possibly degenerated. 
Moreover, $P$ is unique up to rigid motion and reflection in $\R^3$.
\end{lm}

In some sense opposite to Alexandrov's Gluing Theorem is the operation of unfolding.
The first two general methods known to unfold the surface $P$ of any convex polyhedron to a simple (non-overlapping) polygon in the plane 
are the source unfolding and the star unfolding, both with respect to a point $x \in P$.

Concerning the \emph{source unfolding}, one cuts $P$ along the cut locus of the point $x$; this has been studied
for polyhedral convex surfaces since~\cite{SS} (where the cut locus is called the ``ridge tree'').

Concerning the \emph{star unfolding}, one cuts  $P$ along the shortest paths (supposed unique) from $x$ to every vertex of $P$.
The idea goes back to Alexandrov \cite{code2}; the fact that it unfolds $P$ to a non-overlapping polygon was established in \cite{ao}.
\bigskip

An \emph{isosceles tetrahedron} is a a tetrahedron whose opposite edges are pairwise equal.
We shall make use of these tetrahedra and of their special properties, see e.g. \cite{Leech}.

\begin{lm}
\label{acute}
For any isosceles tetrahedron, the total angle at each vertex is precisely $\pi$
and its faces are acute triangles.

Consequently, the star unfolding of an isosceles tetrahedron with respect to any of its vertices provides an acute planar triangle. 
\end{lm}

Some extremal cases in our inequalities are attained by what we call \emph{$\e$-thick tetrahedra}.
Such a tetrahedron is, by definition, a tetrahedron $T$ with one edge included
 in a ball of radius $\varepsilon {\rm diam}(T)$ centered at the midpoint of its longest edge. 
An $\e$-thick tetrahedron is said to be \emph{normal} if its longest edge
and the one opposite to it are, on the one hand, normal to each other, and on
the other hand, normal to the line through their midpoints.

\bigskip

For $x \in S$ put ${\rm rad} _x = \max_{y\in S} ||x-y||$, hence ${\rm rad} _x \geq {\rm rad} (S)$.
Also, put ${\rm Rad} _x = \max_{y\in S} \rho(x,y)$, hence ${\rm Rad} _x \geq {\rm Rad} (S)$.

Denote by $f_x$ the set of all extrinsic farthest points from $x \in S$; i.e., 
$f_x=\{y\in S: ||x-y||= {\rm rad} _x \}$.
Also, denote by $F_x$ the set of all intrinsic farthest points from $x \in S$, and call them \emph{antipodes of} $x$; i.e., 
$F_x=\{y\in S: \rho(x,y)= {\rm Rad} _x \}$.

%%%%%%%%%%%%%%%%%%%%%%%%%%%%%%%%%%%%%%%%%%%%%%%%%

\section{Diameters}
\label{del}

A very nice and deep result of J. O'Rourke and C. A. Schevon \cite{osc} states the following:
if the points $x,y$ in the polyhedral convex surface $P$ realize the intrinsic diameter of $P$
then at least one of them is a vertex of $P$, or they are joined by at least five distinct geodesic segments.
For tetrahedra it directly implies the next lemma.

\begin{lm}
\label{L1}
If $T\in\mathcal{T}$ and $x\in T$ is a point with $\mathrm{Rad}_x =\mathrm{ Diam}(T)$ then $x$ is either a vertex or an antipode of a vertex.
\end{lm}

The extrinsic analog of O'Rourke and Schevon's criterion for diametral points is a simple result, of some interest in itself.

\begin{prop}
\label{LED}
Let $x$, $y\in P$.

(1) If $y\in f_x$ then $y$ is a vertex of $P$. In particular, if $\left\Vert x-y\right\Vert =\mathrm{diam}\left( P\right)$ then both points are vertices of $P$.
 
(2) If $f_x=\{y\}$ and $\left\Vert x-y\right\Vert =\mathrm{rad}\left( P\right)$ then $x$ is the foot of $y$ onto a face.
\end{prop}

\begin{proof}
(1) Assume that $y$ is not a vertex of $P$. Then there exists some line segment $\left[ uv\right]$ on $P$ containing $y$ in its relative interior. Since $\angle xyu+\angle xyv=\pi$, one of these two angles, say $\angle xyu$, is at least $\pi/2$, whence $\left\Vert x-y\right\Vert<\left\Vert x-u\right\Vert$, in contradiction with $y\in f_x$.

(2) Assume now that $\mathrm{rad}\left( P\right)=\left\Vert x-y\right\Vert$ and 
$f_x=\{y\}$, \ie, $\left\Vert x-y\right\Vert>\left\Vert x-v\right\Vert$ for any vertex $v\ne y$. 
By continuity and (1), there is a neighbourhood $N$ of $x$ such that,
for any point $z\in N$, $f_x=\{y\}$.  
Assume now that $x$ is not the foot of $y$ onto a face; then, since $P$ is convex, one can find points $z \in N$ such that $\angle zxy < \pi/2$, and
$\mathrm{rad_z}=\left\Vert z-y\right\Vert<\left\Vert z-x\right\Vert=\mathrm{rad}(P)$, a contradiction.
\end{proof}

\bigskip

In the above Proposition, if $||x-y|| =\mathrm{rad}\left( P\right)$ for $y \in f_x$ and $f_x$ contains at least two points then $x$ may not be the foot of $y$ onto a face.
For example, consider a normal $\varepsilon$-thick tetrahedron, with $y,z$ the vertices of the longest edge and $x$ its mid-point.

\begin{co}
\label{reg_ext}
For the regular tetrahedron $T$ of unit edge, ${\rm diam} (T) = 1$ and ${\rm rad} (T) = \sqrt{\frac23}$.
\end{co}

Corollary \ref{reg_ext} is the extrinsic analog of Theorem 3.1 in \cite{ro1}, 
quoted in the  next lemma to clarify the second equality case in Theorem \ref{TE1}.

\begin{lm}
\label{reg_int}
For the regular tetrahedron $T$ of unit edge, ${\rm Diam} (T) =  \frac2{\sqrt{3}}$ is realized between any vertex and the centre of its opposite face,
while ${\rm Rad} (T) = 1$ is realized between mid-points of opposite edges.
\end{lm}

Lemma \ref{basic} immediately implies the next one.
An {\it Y-tree} is a tree with one junction point and three edges.

\begin{lm}
\label{L2}
The cut locus of a vertex of $T\in\mathcal{T}$ is a (possibly degenerate) $Y$-tree.
\end{lm}

We need one more lemma for our first main result.

\begin{lm}
\label{eq_tri}
Consider the family $\mathcal{I}$ of all planar acute triangles $\Delta$ inscribed in a given circle $\mathcal{C}$, 
and let $l_\Delta$ denote the longest side of $\Delta \in \mathcal{I}$.
Then $\inf_{\Delta \in \mathcal{I}} l_\Delta$ is achieved for equilateral triangles.
\end{lm}

\begin{proof}
Under the hypotheses, just note that the edge lengths of the triangle are in the same order as the lengths of the intercepted arc of circles,
by the Sine Rule and the monotony of the sine function on $[0,\pi/2]$. 
The longest arc is obviously shortest when all three arcs are equal.
\end{proof}

\begin{thm}
\label{TE1}
For any tetrahedron holds
\[
1\leq\frac{\mathrm{Diam}\left( T\right) }{\mathrm{diam}\left( T\right)
}\leq\frac{2}{\sqrt{3}}
\]
and both inequalities are sharp. 

The first inequality becomes equality, for example, for $\e$-thick tetrahedra with $\varepsilon$ small enough, 
while the second inequality becomes equality for the regular tetrahedon.
\end{thm}

\begin{proof}
The intrinsic distance between two points is never less than the extrinsic distance, so the first inequality is obvious. 
It is also clear that both diameters are equal for an $\e$-thick tetrahedron, whenever $\varepsilon>0$  is sufficiently small.

By Lemma \ref{L1}, there exists a vertex $v\in T$ and a point $p\in T$ such
that $\rho\left( v,p\right) =\mathrm{Diam}\left( T\right)$. If $p$ is
also a vertex, then the $[pv]$ is greater than or equal to any
other edge, for the lengths of edges are also the intrinsic distance between
their endpoints. Hence, by Lemma \ref{LED}, 
$\mathrm{Diam}\left( T\right)=\mathrm{diam}\left( T\right)$.

Assume now that $p$ is a flat point, that is, the only triple point of
$C\left( v\right)  $, which is a Y-tree be virtue of Lemma \ref{L2}. Cutting
$T$ along the three edges meeting at $v$ and unfolding it onto a plane yields
the development shown in plain lines in Figure 1. 
There, $p$ is on the bisector line of any two of the images of $v$, 
by a direct consequence of Lemma \ref{basic} (iv).
Hence $p$ is the circum-centre of the triangle determined the three images of $v$.
We now conside the triangle $\Delta$ drawn in dots, with vertices at the images of $v$, 
as the unfolding of an isosceles tetrahedron denoted by $T^{\prime}$, see Lemma \ref{gluing}.
Notice that $T'$ has acute triangles as faces, by Lemma \ref{acute}, even though $\Delta$ might not be acute.

We claim that each edge of $T'$ is shorter than one edge of $T$.
This is clear if $\Delta$ is acute, because then the edges of $T'$ have half-length of the sides of $\Delta$, 
which in turn have a length less than twice an edge of T by the triangle inequality.
If $\Delta$ is not acute, then four of the edges of $T'$ have half-length of two sides of $\Delta$.
The last two edges of $T'$ equal the length of the median line of $\Delta$ with respect to its longest edge, 
and so, since $\Delta$ is not acute, it is strictly shorter than half-length of that side.

Whence $\mathrm{diam}\left( T^{\prime}\right) \leq\mathrm{diam}\left( T\right)$. 
On the other hand, the intrinsic distance between $p$ and $x$ is unchanged, whence $\mathrm{Diam}
\left( T^{\prime}\right)  \geq\rho\left( x,v\right) =\mathrm{Diam}\left(T\right) $.

\begin{figure*}
\label{F3}
\centering
 \includegraphics[width=\textwidth]{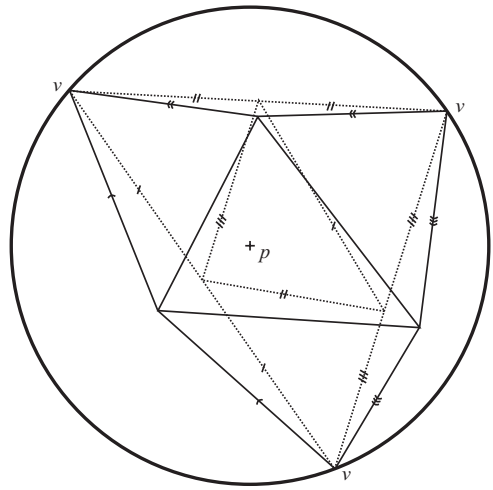}
\caption{Unfoldings of $T$ and $T^{\prime}$.}
\end{figure*}

By Lemma \ref{L1}, there exists a vertex $v^{\prime}\in T^{\prime}$ and a point
$p^{\prime}\in T^{\prime}$ such that $\rho\left(  v^{\prime},p^{\prime
}\right)  =\mathrm{Diam}\left(  T^{\prime}\right)  $. Cutting $T^{\prime}$
along the three edges meeting at $v^{\prime}$ and unfolding it onto a plane
yields an acute triangle (see Lemma \ref{acute}) similar to one the shown in doted lines in Figure 1.
It is now easy to see that the longest edge of $T^{\prime}$ is longer that the
edge of a regular tetrahedron $T^{\prime\prime}$ whose unfolding is
inscribed in same the circle, see Lemma \ref{eq_tri}. Moreover, this deformation doesn't change the distance 
$\mathrm{Diam}\left( T^{\prime}\right) =\rho\left( p^{\prime},v^{\prime}\right) 
\leq\mathrm{diam}\left( T^{\prime\prime}\right)$,
whence
\[
\frac{\mathrm{Diam}\left(  T\right)  }{\mathrm{diam}\left(  T\right)  }%
\leq\frac{\mathrm{Diam}\left(  T^{\prime}\right)  }{\mathrm{diam}\left(
T\right)  }\leq\frac{\mathrm{Diam}\left(  T^{\prime\prime}\right)
}{\mathrm{diam}\left(  T^{\prime\prime}\right)  }=\frac{2}{\sqrt{3}}\text{.}%
\]
\end{proof}

\bigskip

The upper bound given by Theorem \ref{TE1} for tetrahedra, $\frac{2}{\sqrt{3}} \approx 1.15$, is clearly better than 
the upper bound obtained by N. P. Makuha \cite{ma} for general convex surfaces, $\frac{\pi}{2} \approx 1.57$.

%%%%%%%%%%%%%%%%%%%%%%%%%%%%%%%%%%%%%%%%%%%%%%%%%

\section{Diameters and radii}
\label{radi}

It follows from the triangle inequality, in any compact metric space, that the ratio between diameter and radius belongs to $[1,2]$. 
In the case of the intrinsic metric of a tetrahedron, we have the following result.

\begin{prop}
\label{Diam/Rad}
For any convex polyhedron $P$ holds
\[
1 < \frac{\mathrm{Diam}\left( P \right) }{\mathrm{Rad}\left( P\right) }\leq2 \text{.}
\]

The second inequality is sharp, and achieved by normal $\e$-thick tetrahedra, for small $\e$.
\end{prop}

\begin{proof}
For each convex surface $S$ with $\mathrm{Rad}\left(S\right)=\mathrm{Diam}\left( S\right)$, the mapping $F$ is a single-valued involution \cite{vz1}, 
and no convex polyhedron has this property \cite{JR}, hence $\mathrm{Rad}\left(S\right)< \mathrm{Diam}\left( S\right)$ in this case.

Consider now a normal $\e$-thick tetrahedron $T$. 
Let $a$, $b$ be the
endpoints of its longest edge and $c$, $d$ the two other vertices. Let $m$ be the
midpoint of $[ab]$, and let $q$ be the midpoint of $\left[ cd\right]$. Since $T$ is symmetric with
respect to the plane $abq$, the Jordan arc of $C\left( m\right)$ between
$a$ and $b$ should be included in this plane, and so is the union of
$\left[aq\right]$ and $\left[ bq\right]$. Similarly, the symmetry with respect
to the plane $mcd$ infers that the Jordan arc of $C\left( m\right) $ between
$c$ and $d$ is $\left[  cd\right]  $. It follows that $q$ is the only point in
$C\left(  m\right)  $ of degree more than two. Hence
$F_{m}\subset\left\{  q,a,b,c,d\right\}  $. It is clear that for $\varepsilon$
small enough $\rho\left(  m,q\right)  $ and $\rho\left( m,c\right)
=\rho\left( m,d\right)$ are both less that $\rho\left(  m,a\right)
=\rho\left( m,b\right) =\frac{\nm{a-b}}{2}$, whence $\mathrm{Rad}_{m}=\frac{\nm{a-b}}2=\mathrm{Rad}\left( T\right)$. 
This completes the proof.
\end{proof}

\bigskip

We have a similar result for the extrinsic metric.

\begin{prop}
For any  convex polyhedron $P$  holds
\[
1 < \frac{\mathrm{diam}\left( P\right) }{\mathrm{rad}\left( P\right) } \leq 2 \text{.}
\]
The second inequality is sharp, and achieved by $\e$-thick tetrahedra, for small $\e$.
\end{prop}

\begin{proof}
The first inequality is strict because no polyhedron has constant width.

The case of equality is easy to check.
\end{proof}

\bigskip

From Theorem 1 and Proposition 3 directly follows

\begin{co}
For any  tetrahedron $T$  holds
\[
\frac{\sqrt{3}}{4} < \frac{\mathrm{rad}}{\mathrm{Diam}} < 1 \text{.}
\]
\end{co}

We return now to the statement of Proposition \ref{TE1}.

\begin{conj}
\label{conj}
Prove that
\[
\frac2{\sqrt{3}}\leq\frac{\mathrm{Diam}\left( T\right) }{\mathrm{Rad}\left( T\right) }\text{,}
\]
with equality for the regular tetrahedon.
\end{conj}

Notice that, if solved, the above problem and Theorem \ref{TE1}
would imply $\mathrm{Rad}\left( T\right) \leq\mathrm{diam}\left( T\right)$, 
with equality for the regular tetrahedron. However this latter inequality
can be proven directly.

\begin{thm}
\label{Tder}
For any tetrahedron $T\in \mathcal{T}$ we have%
\[
\mathrm{Rad}\left( T\right) \leq\mathrm{diam}\left( T\right) \text{,}%
\]
with equality if and only if $T$ is regular.
\end{thm}

\begin{proof}
Let $T$ be a tetrahedron of unit extrinsic diameter. Denote by $a$ and $b$ the
endpoints of its (or one of its) longest edge(s), and by $c$ and $d$ the two
other vertices. Let $o$ be the midpoint of $\left[  ab\right] $. Cutting
along the three edges meeting at $d$ and unfolding $T$ onto a plane yields the
development $G_0$ shown in plain lines in Figure 2, where $d_a, d_b$ and $d_c$ are the images of $d$. 

\begin{figure*}
\label{Fder}
\centering
 \includegraphics[width=\textwidth]{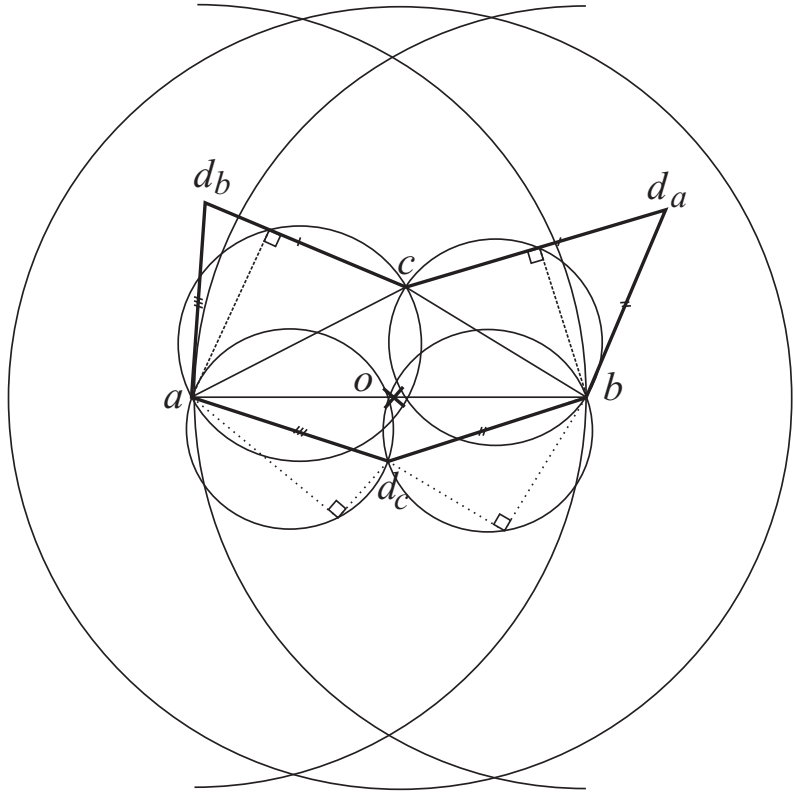}
\caption{The development of $T$ in the proof of Theorem \ref{Tder}.}
\end{figure*}

If the triangle $bcd_a$ is right or obtuse at $d_a$ then it is included in the disc of diameter $\left[  bc\right]$, and we set $G=G_0$.

If the triangle $bcd_a$ is right or obtuse at $c$ then cut along $[bc]$ and rotate it around $b$, 
to obtain another development $G$ in which the image of the face $bcd$ is included in the union of the discs of diameters 
$\left[  bc\right]  $ and $\left[bd_c\right]$.

Otherwise, if  the triangle $bcd_a$  is acute at $d_a$ then cut it along the dash line and rotate the outer part
around $b$, to obtain another development $G$ in which the image of the face $bcd$ is included
in the union of the discs of diameters $\left[  bc\right]  $ and $\left[bd_c\right] $. 
In the same way, we can arrange the face $acd$ in the union
of the discs of diameters $\left[  ac\right] $ and $\left[  a d_c\right] $. 

We claim that each of those four discs are included in the disc $D$ of center
$o$ and radius $1$. Let $u$ be the midpoint of $\left[  ac\right]  $. 
Since $[ac]$ is less than or equal to the longest edge of $T$, $c$
belongs to the disc of center $a$ and radius $1$, and thus $u$ belongs to the
disc of diameter $\left[  ab\right]$:
$||a-u||=||a-c||/2\leq\frac{1}{2}$, $||o-u||\leq\frac{1}{2}$. 
Hence the disc of diameter $\left[  ac\right]$ is included in $D$. The proof is similar for the three other discs.
We shall use this fact to conclude the equality case.

Notice that the quadrilateral $acbd_c$ is convex, because $[ab]$ is the longest edge in $T$.

Let $o'$ belong to $F_{o}$; for simplicity, also denote by $o'$ the image
of $o'$ on $G$ (or one of its images, if there are more).

Assume first that $o'$ belongs to the convex quadrilateral  $acbd_c$.
Then $||o-o'|| \leq \max \{ ||o-a||, ||o-c||, ||o-b||, ||o-d_c|| \}$,
and Apollonius's theorem gives $ \max \{ ||o-c||, ||o-d_c|| \} \leq \sqrt3/2 <1$.

Assume now that $o'$ belongs to one of the four discs 
of diameters the sides of the quadrilateral $acbd_c$, say the one centered at $u$.
Then 
$$\rho(o,o') \leq ||o-u|| + ||u-o'|| \leq |b-c||/2 + ||a-c||/2 \leq 1.$$

It follows that $\mathrm{Rad}\left(  T\right)  \leq \mathrm{Rad}_{o}=\rho \left(  o,o' \right)  \leq1$.

Assume now that we have equality, hence 
$\mathrm{Rad}\left(  T\right)  = \mathrm{Rad}_{o}=\rho \left(  o,o^{\prime}\right) = 1$. 
The development has to intersect $\partial D$. 
Notice that there are at least three geodesic segments joining $o$ to $o' \in F_o$ on $T$.
So it follows, moreover, that three of the four small circles have to be tangent to $\partial D$.
Assume one of them is that of diameter $\left[  ac\right]$. 
This implies that $c$ is one point of intersection of the circles of radius one
centered at $a$ and $b$, and so $||a-c||=||b-c||=||a-b||=1$. 
Similarly  $||a-d_c||=||b-c||=||b-d_c||=1$. So all edges, except possibly $\left[  cd\right]$, have length one. 

Now one can repeat the whole argument above replacing $\left[  ab\right]  $ by another longest edge to
prove that $T$ is actually regular. 
\end{proof}

\bigskip

The first inequality below is valid for arbitrary convex surfaces; compare it to the first one in Theorem 1.

\begin{prop}
For any tetrahedron $T\in \mathcal{T}$  holds
\[
1\leq\frac{\mathrm{Rad}\left( T\right) }{\mathrm{rad}\left( T \right)} < 2.
\]

The first inequality is sharp, and achieved by normal $\e$-thick tetrahedra, for small $\e$.
\end{prop}

\begin{proof}
We have, for all $x,y \in S$, $|| x-y ||\le\rho(x,y)$, so $\rad_x\le\Rad_x$, hence $\rad\left( S\right) \le\Rad \left( S\right)$.

The case of equality is easy to check, see the proof of Proposition \ref{Diam/Rad}.

Theorem \ref{Tder} and  Proposition 3 yield the last inequality.
\end{proof}

%%%%%%%%%%%%%%%%%%%%%%%%%%%%%%%%%%%%%%%%%%%%%%%%%%%%%%%%%%%%%%

\medskip

Jin-ichi Itoh

\noindent{\footnotesize School of Education, Sugiyama Jogakuen University
\newline 17-3 Hoshigaoka-motomachi, Chikusa-ku, Nagoya, 464-8662 Japan}

{\small \hfill j-itoh@sugiyama-u.ac.jp}

\medskip

Jo\"{e}l Rouyer

\noindent{\footnotesize 16, rue Philippe, 68 200 H\'egenheim - France

{\small \hfill Joel.Rouyer@ymail.com}

\medskip

Costin V\^\i lcu

\noindent{\footnotesize {\sl Simion Stoilow} Institute of Mathematics of the Roumanian Academy
\newline P.O. Box 1-764, 014700 Bucharest, Roumania}

{\small \hfill Costin.Vilcu@imar.ro}


\begin{thebibliography}{99}

\bibitem{code1} 
P. K. Agarwal, B. Aronov, J. O'Rourke and C. A. Schevon, 
{\it Star unfolding of a polytope with applications}, 
SIAM J. Comput. {\bf 26} (1997), 1689-1713

\bibitem{code2} 
A. D. Alexandrov, 
{\it Convex Polyhedra}, 
Springer-Verlag, Berlin, 2005. Monographs in Mathematics. 
Translation of the 1950 Russian edition by N. S. Dairbekov, S. S. Kutateladze, and A. B. Sossinsky

\bibitem{ao}  B. Aronov and J. O'Rourke,
{\it Nonoverlap of the star unfolding},
Discrete Comput. Geom. {\bf 8} (1992), 219-250.

\bibitem{cg} G. D. Chakerian and H. Groemer, 
\textit{Convex Bodies of Constant Width}, 
in \textsl{Convexity and its Applications}, 
P. Gruber and J. Wills (Eds.), Birkh\"auser, Bassel 1983, 49-96

\bibitem{dty}  V. Dods, C. Traub, J. Yang, {\it Geodesics on the regular tetrahedron and the cube}, 
Discrete Math. {\bf 340} (2017), 3183-3196

\bibitem{Leech}  J. Leech, {\it Some Properties of the Isosceles Tetrahedron}, 
The Mathematical Gazette {\bf 34} (1950), 269-271

\bibitem{ma}  N. P. Makuha, 
{\it A connection between the inner and the outer diameters of a general closed convex surface} (in Russian),
Ukrain. Geometr. Sb. Vyp. {\bf 2} (1966), 49-51

\bibitem{osc}  J. O'Rourke and C. A. Schevon,
{\it Computing the geodesic diameter of a 3-polytope},
Proc. 5th ACM Symp. Comput. Geom. (1989), 370-379

\bibitem{ro1}  J. Rouyer, 
{\it Antipodes sur un t\'etra\`edre r\'egulier}, 
J. Geom. {\bf 77} (2003), 152-170

\bibitem {JR}J. Rouyer, 
\emph{Steinhaus conditions for convex polyhedra}, in: 
 K. Adiprasito et al. (Eds.), {\it Convexity and Discrete Geometry Including Graph Theory}, 
 Springer Proceedings in Mathematics \& Statistics 148,
Springer International Publishing 2016, 77--84

\bibitem{Ro-V}  J. Rouyer and C. V\^\i lcu,
{\it Sets of tetrahedra, defined by maxima of distance functions}, 
An. \c Stiin\c t. Univ. ``Ovidius'' Constan\c ta Ser. Mat. {\bf 20} (2012), 197-212.

\bibitem{SS}  M. Sharir and A. Schorr, 
{\it On shortest paths in polyhedral spaces}, SIAM J. Comput. {\bf 15} (1986), 193-215.

\bibitem{vz1}  C. V\^\i lcu and T. Zamfirescu,
{\it Symmetry and the farthest point mapping on convex surfaces},
Adv. Geom. {\bf 6} (2006), 345-353

\bibitem{vz2}  C. V\^\i lcu and T. Zamfirescu,
{\it Multiple farthest points on Alexandrov surfaces}, 
Adv. Geom. \textbf{7} (2007), 83-100

\bibitem{Zalgaller} V. A. Zalgaller, \textit{An isoperimetric problem for tetrahedra},
J. Math. Sci. \textbf{140} (2007), 511-527

\end{thebibliography}
\end{document}